\documentclass[]{article}

\usepackage{verbatim}
\usepackage{amssymb}
\usepackage{amsbsy}
\usepackage{amscd}
\usepackage{amsmath}
\usepackage{amsthm}
\usepackage[mathscr]{eucal}

\newtheorem{theorem}{Theorem}
\numberwithin{defn}{section}

\usepackage{booktabs}

\newcommand{\bfx}{{\mathbf x}}

\newcommand{\bbR}{{\mathbb R}}
\newcommand{\bbP}{{\mathbb P}}
\newcommand{\bbZ}{{\mathbb Z}}
\newcommand{\bbS}{{\mathbb S}}

\newcommand{\calA}{\mathcal{A}}
\newcommand{\calS}{\mathcal{S}}
\newcommand{\calM}{\mathcal{M}}
\newcommand{\calB}{\mathcal{B}}
\newcommand{\calN}{\mathcal{N}}
\newcommand{\calV}{\mathcal{V}}

\newcommand{\dS}{{\mathfrak S}}

\newcommand{\ch}{{\rm \bf Ch}}
\newcommand{\Dch}{{\rm \bf D}}
\newcommand{\im}{\operatorname{im}}

\newtheorem{remark}{Remark}
\newtheorem{prop}{Proposition}
\newtheorem{lemma}{Lemma}

\title{Application of arrangement theory to unfolding models}

\author{Hidehiko Kamiya, Akimichi Takemura and Norihide Tokushige
}

\date{September, 2011}

\begin{document}

\maketitle

\begin{abstract}
Arrangement theory plays an essential role in 
the study of the unfolding model used in many fields. 
This paper describes how arrangement theory can be 
usefully employed in solving the problems of counting 
(i) the number of admissible rankings in an unfolding 
model and (ii) the number of ranking patterns generated 
by unfolding models. 
The paper is mostly expository but also contains some 
new results such as simple upper and lower bounds for 
the number of ranking patterns in the unidimensional 
case.  
\end{abstract}

\noindent
{\it Keywords and phrases}: all-subset arrangement, braid arrangement, 
chamber, characteristic polynomial, 
finite field method, hyperplane arrangement, 
intersection poset, 
mid-hyperplane arrangement, partition lattice, 
ranking pattern, unfolding 
model.

\section{Introduction}

The unfolding model 
(Coombs \cite{coo-50}, De Leeuw \cite{del}) 
is a model for preference rankings in psychometrics. 
It is now widely applied not only in psychometrics 
(De Soete, Feger and Klauer \cite{dfk}) but also in 
other fields such as marketing science 
(DeSarbo and Hoffman \cite{deh}) and 
voting theory (Clinton, Jackman and Rivers \cite{cjr}). 
The model is also used as a submodel for 
more complex models, as in item response theory for unfolding 
(Andrich \cite{and-88, and-89}). 
Moreover, in the context of Voronoi diagrams, 
this model can be regarded as a higher-order 
Voronoi diagram (Okabe, Boots, Sugihara and Chiu 
\cite{obsc}). 

The unfolding model describes the ranking process 
in which judges rank a set of objects 
in order of preference. 
In this model, judges and objects are 
assumed to be represented by points in the 
Euclidean space $\bbR^n$. 
Suppose a judge $y\in \bbR^n$ ranks $m$ objects 
$x_1, 
\ldots,x_m \in \bbR^n$.  
According to the unfolding model, $y$ ranks 
$x_1,
\ldots,x_m$ in descending order 
of proximity in the usual Euclidean distance.  
Hence, $y$ likes $x_{i_1}$ best, $x_{i_2}$ second 
best, and so on, iff $\| y-x_{i_1} \|<
\| y-x_{i_2} \|<\cdots <\| y-x_{i_m} \|$. 
In this case, we will say $y$ gives ranking 
$(i_1 i_2 \cdots i_m)$.   

For a given $m$-tuple $(x_1,\ldots,x_m)$ 
of objects, let 
${\rm RP}^{{\rm UF}}(x_1,\ldots,x_m)$ be the set 
of admissible rankings, 
i.e., $(i_1 \cdots i_m)$ such that 
$ \| y-x_{i_1} \| < \cdots < \| y-x_{i_m} \|$ 
for some $y\in \bbR^n$. 
We call ${\rm RP}^{{\rm UF}}(x_1,\ldots,x_m)$ the 
ranking pattern of the unfolding model with 
$m$-tuple $(x_1,\ldots,x_m)$. 
In the psychometric literature, there has not been 
much study on the structure of the ranking pattern. 
In this paper, we investigate the ranking pattern by 
using the theory of hyperplane arrangements 
(Orlik and Terao \cite{ort}). 
Specifically, we consider the following two problems: 
\begin{enumerate}
\item[(i)] 
Find the cardinality of 
${\rm RP}^{{\rm UF}}(x_1,\ldots,x_m)$ for a given 
generic $m$-tuple $(x_1,\ldots,x_m)$; 
\item[(ii)] 
Find the cardinality of 
\[
\{ {\rm RP}^{{\rm UF}}(x_1,\ldots,x_m): 
(x_1,\ldots,x_m) \text{ is a generic $m$-tuple} \}. 
\]
\end{enumerate}
The first problem asks how many rankings are admissible 
in one unfolding model, and the second inquires how many 
ranking patterns are possible by using different 
unfolding models 
(that is, by taking different choices of $m$-tuples 
of objects). 
As we will see, these problems can be reduced to 
those of counting the numbers of chambers of 
some real arrangements; 
moreover, the latter problems can be solved by employing 
general results in the theory of hyperplane arragements 
(e.g., Zaslavsky's result on 
the number of chambers of a real arrangement, 
the finite field method, etc.). 
In this sense, arrangement theory plays an essential 
role in the study of the unfolding model. 

This paper gives a survey of recent results 
(\cite{kott}, \cite{kt-97}, \cite{kt-05}, 
\cite{ktt-10}) 
on the problems stated above. 
It also contains new results on upper and lower bounds 
for the number of ranking patterns in the unidimensional 
case $n=1$. 
In addition, the problem of counting 
inequivalent ranking patterns 
(i.e., those which cannot be obtained from one another 
by just the relabeling of the objects) 
when $n=1$ was not dealt with specifically in 
\cite{kott} 
but is discussed fully in the present paper. 

The organization of the paper is as follows. 
In Section \ref{sec:NAR}, we define genericness of the 
unfolding model, and give the answer to 
problem (i) above, 
i.e., the number of admissible rankings of 
the unfolding model with generic objects. 
Next, in Section \ref{sec:NRP} we discuss the problem of 
counting the number of ranking patterns (problem (ii)). 
In Subsection \ref{subsec:UUM}, we deal with the 
unidimensional case, and give the number of ranking 
patterns in terms of the number of chambers of 
the mid-hyperplane arrangement. 
We also provide explicit upper and lower bounds for the 
number of ranking patterns. 
In Subsection \ref{subsec:UMCO}, we treat the unfolding 
model of codimension one, where the restriction by 
dimension is weakest. 
In this case, we describe how  
the number of ranking patterns can be expressed by 
the number of chambers of an arrangement 
called 
the all-subset arrangement. 

\section{Number of admissible rankings}
\label{sec:NAR} 

In this section, we define genericness of the 
unfolding model, and discuss the problem of 
counting the number of admissible rankings 
generated by the unfolding model with generic objects. 

Suppose we are given 
$x_1,\ldots,x_m \in \bbR^n$ with 
$m\ge 3$ and 
$n\le m-2$. 

In general, for $m$ distinct points 
$z_{1},\ldots,z_{m} \in \bbR^{\nu} \ (m \ge \nu+1)$, 
let $\overline{z_{i}z_{j}}$ 
denote the one-simplex connecting two points
$z_{i}$ and $z_{j} \ (i < j)$.
Consider the following condition:

\begin{enumerate}
\item [(A)] 
The union of $\nu$ distinct one-simplices
$\overline{z_{i_{k} } z_{j_{k} } } $ 
$(i_{k} < j_{k}, \ k =1,\ldots,\nu)$ contains no loop
if and only if the corresponding vectors 
$z_{i_{k} } - z_{j_{k} } $ 
$(k =1,\ldots,\nu)$ 
are linearly independent.
\end{enumerate}

We assume $x_1,\ldots,x_m\in \bbR^{n} \ (n\le m-2)$ are 
generic in the sense that they satisfy 
the following two conditions: 
\begin{enumerate}
\item [(A1)] 
The $m$ points $x_1,\ldots,x_m\in \bbR^{n}$ 
satisfy condition (A).
\item [(A2)] The $m$ points
$
(x_1^T, \|x_{1} \|^{2})^T,   
\ldots,
(x_m^T, \|x_{m} \|^{2})^T
\in \bbR^{n+1}$
satisfy condition (A).
\end{enumerate} 

Now, according to the unfolding model, 
judge $y \in \bbR^n$ prefers $x_i$ to 
$x_j$ ($i\ne j$) iff 
$\| y-x_i \| < \| y-x_j \|$. 
This condition is equivalent to $y$ being on the 
same side as $x_i$ of the perpendicular bisector 
\begin{eqnarray*}
H_{ij}
&:=& \{ y\in \bbR^n: \| y-x_i \|=\| y-x_j \| \} \\ 
&=& \{ y\in \bbR^n: 
(x_i-x_j)^T(y-\frac{x_i+x_j}{2})=0 \} 
\end{eqnarray*} 
of the line segment $\overline{x_ix_j}$ 
joining $x_i$ and $x_j$.    
Let us define a hyperplane arrangement 
\[
\calA_{m,n}=\calA_{m,n}(x_1,\ldots,x_m)
:=\{ H_{ij}: 1 \le i<j \le m\} 
\]
in $\bbR^n$. 
We call $\calA_{m,n}$ the unfolding arrangement. 

Then $\calA_{m,n}$, 
like any real hyperplane arrangement, cuts $\bbR^n$ 
into chambers, i.e., 
connected components of the complement 
$\bbR^n \setminus \bigcup \calA_{m,n}$, where 
$\bigcup \calA_{m,n}:=\bigcup_{H 
\in \calA_{m,n}}H$.  
Moreover, each of these chambers is of the form 
\[
C_{i_1\cdots i_m}:=
\{ \| y-x_{i_1} \| 
<\cdots <\| y-x_{i_m} \|\} 
\ne \emptyset
\]
for some admissible ranking $(i_1 
\cdots i_m)\in \bbP_m$, 
where $\bbP_m$ denotes the set of permutations of 
$[m]:=\{ 1,\ldots,m\}$. 

We observe that $y\in \bbR^n$ gives ranking 
$(i_1\cdots i_m)\in \bbP_m$ if and only if 
$y \in C_{i_1\cdots i_m}\ne \emptyset$. 
Thus there is a one-to-one correspondence between 
the set of admissible rankings and the set of chambers 
$\ch(\calA_{m,n})$ of $\calA_{m,n}$: 
\[
(i_1\cdots i_m) \leftrightarrow C_{i_1\cdots i_m} 
\]
for $(i_1\cdots i_m)$ such that 
$C_{i_1\cdots i_m}\ne \emptyset$. 
This implies that the problem of counting the number of 
admissible rankings reduces to that of counting the 
number of chambers of $\calA_{m,n}$. 
The answer to the latter problem is given by the 
theorem below. 
Let $\calS^m_k \ (k \in \bbZ)$ be the 
signless Stirling numbers of the first kind: 
$t(t+1)\cdots (t+m-1)=\sum_k \calS^m_kt^k$.  

\begin{theorem}[Good and Tideman \cite{got}, 
Kamiya and Takemura \cite{kt-97,kt-05}, 
Zaslavsky \cite{zas-02}]
\label{th:|ch(A)|}
Suppose $x_1,\ldots,x_m \in \bbR^n \ (n\le m-2)$ 
are generic. 
Then, the number of chambers of 
$\calA_{m,n}=\calA_{m,n}(x_1,\ldots,x_m)$ is  
\[
|\ch(\calA_{m,n})|
=\calS^m_{m-n}+\calS^m_{m-n+1}+\cdots +\calS^m_{m}. 
\]
Furthermore, the number of bounded chambers of 
$\calA_{m,n}$ 
is 
\[
\calS^m_{m-n}-\calS^m_{m-n+1}+\calS^m_{m-n+2}
-\cdots +(-1)^n\calS^m_{m}. 
\]  
\end{theorem}

The proof of Theorem \ref{th:|ch(A)|} is based on 
Zaslavsky's general result on the number of 
chambers of an arrangement (Zaslavsky \cite{zas-75}) 
and the following proposition. 
Denote by $\Pi_m$ the partition lattice, consisting of 
partitions of $[m]$ and ordered by refinement. 
Further, let $\Pi_m^n$ stand for the rank 
$n$ truncation of $\Pi_m$, 
i.e., the subposet of $\Pi_m$ comprising elements 
of rank ($=m-\text{ \# of blocks}$)
at most $n$.   

\begin{prop}[Kamiya and Takemura \cite{kt-97,kt-05}]
The intersection poset $L(\calA_{m,n})$ of the 
unfolding arrangement $\calA_{m,n}$ is isomorphic to 
$\Pi_m^n$: 
\[
L(\calA_{m,n})\cong \Pi_m^n.
\]
\end{prop}

The isomorphism is given by 
\[
L(\calA_{m,n})\ni X \mapsto I_X \in \Pi_m^n, 
\]
where $I_X$ 
is the partition of $[m]$ into equivalence classes 
under the equivalence relation $\sim_X$ defined by  
$i \sim_X j \stackrel{{\rm def}}{\iff} X \subseteq 
H_{ij}$ ($H_{ii}:=\bbR^n$). 

\begin{remark}
When $n\ge m-1$, and $x_1,\ldots,x_m \in \bbR^n$ 
satisfy condition (A1) with the $\nu=n$ in (A) 
replaced by $m-1$, 
we can easily see that $|\ch(\calA_{m,n})|=m!$ 
and that the number of 
bounded chambers of $\calA_{m,n}$ is zero 
(so the results in Theorem \ref{th:|ch(A)|} continue 
to be valid).  
Therefore, all $m!$ rankings arise as unbounded chambers of 
$\calA_{m,n}$ in this case.  
\end{remark}

\section{Number of ranking patterns}
\label{sec:NRP} 

In this section, we consider the problem of 
counting the number of ranking patterns. 
We treat two extreme cases---the unidimensional 
unfolding model: $n=1$ (Subsection \ref{subsec:UUM}) 
and the unfolding model of 
codimension one: $n=m-2$ (Subsection \ref{subsec:UMCO}). 

\subsection{Unidimensional unfolding models} 
\label{subsec:UUM} 

In this subsection, we look into 
the problem of counting 
the number of ranking patterns of  
unidimensional unfolding models: $n=1$. 
A related problem is studied in Stanley \cite{sta-06}. 

In this case $n=1$, 
objects are $m$ points on the real line: 
$x_1,\ldots,x_m\in \bbR$. 
We assume $x_1,\ldots,x_m$ are generic, i.e., 
the midpoints $x_{ij}:=(x_i+x_j)/2, \ 1\le i<j \le m$, 
are all distinct.  
This condition can be written as 
\[
(x_1,\ldots,x_m) 
\in \bbR^m\setminus \bigcup \calM_m, 
\]
where 
$\calM_m:=\calB_m \cup \calN_m$ 
is the mid-hyperplane arrangement 
(Kamiya, Orlik, Takemura and Terao \cite{kott})  
with 
\begin{eqnarray*}
\calB_m &:=& \{ K_{ij}: 1\le i<j \le m \}, 
\quad 
K_{ij}:=\{ (x_1,\ldots,x_m)\in \bbR^m: x_i=x_j \}, \\ 
\calN_m &:=& \{ H_{ijkl}: (i,j,k,l)\in I_4 \}, \\ 
H_{ijkl} &:=& \{ (x_1,\ldots,x_m)\in \bbR^m: 
x_i+x_j=x_k+x_l \}, \\ 
I_4 &:=& \{ (i,j,k,l): i,j,k,l 
\text{ are all distinct, } \\ 
&& \qquad \qquad \qquad \qquad \qquad 
1\le i<j \le m, \ i<k<l\le m \}.   
\end{eqnarray*}
(In this subsection, we write elements of $\bbR^m$ as 
row vectors.) 
Note that $\calB_m$ is the braid arrangement. 
We have $H_{ij}=\{ x_{ij}\}, \ 
1\le i<j \le m$,     
and $\calA_{m,1}=\{ \{ x_{ij}\}: 1\le i<j \le m\}$. 

An $m$-tuple $\bfx:=(x_1,\ldots,x_m)$ of objects 
gives the ranking pattern 
\[
{\rm RP}^{{\rm UF}}(\bfx)=\{ (i_1\cdots i_m)
\in \bbP_m: |y-x_{i_1}|<\cdots <|y-x_{i_m}| 
\text{ for some } y \in \bbR\}. 
\] 
We want to know 
\begin{equation}
\label{eq:def-r(m)}
r(m):=|\{ {\rm RP}^{{\rm UF}}(\bfx): \bfx 
\in \bbR^m\setminus \bigcup \calM_m \}|. 
\end{equation}

The braid arrangement $\calB_m$ has a chamber 
$C_0\in \ch(\calB_m)$ defined by $x_1<\cdots <x_m$: 
\[
C_0:=\{ (x_1,\ldots,x_m)\in \bbR^m: x_1<\cdots <x_m \}. 
\]
Let us concentrate our attention on $C_0$. 
For $\bfx=(x_1,\ldots,x_m)
\in C_0 \setminus \bigcup \calN_m$ 
and $\bfx'=(x'_1,\ldots,x'_m)
\in C_0 \setminus \bigcup \calN_m$, 
we can easily see that 
${\rm RP}^{{\rm UF}}(\bfx)=
{\rm RP}^{{\rm UF}}(\bfx')$ if and only if 
the order of the midpoints on $\bbR$ is the same for 
$\bfx$ and $\bfx'$ 
(i.e., $\forall (i,j,k,l)\in I_4: 
x_{ij}<x_{kl} \iff x'_{ij}<x'_{kl}$).  
Noting that $x_{ij}<x_{kl}$ iff $(x_1,\ldots,x_m)
\in H_{ijkl}^-:=\{ (x_1,\ldots,x_m)\in\bbR^m: 
x_i+x_j<x_k+x_l \}$,  
we obtain the following lemma. 
 
\begin{lemma}[Kamiya, Orlik, Takemura and Terao \cite{kott}]
\label{lm:RP=RP}
For $\bfx,\bfx'\in C_0 \setminus \bigcup \calN_m$, 
we have ${\rm RP}^{{\rm UF}}(\bfx)=
{\rm RP}^{{\rm UF}}(\bfx')$ if and only if 
$\bfx$ and $\bfx'$ are in the same chamber of 
$\calN_m$.  
\end{lemma}

Put 
\[
r_0(m):=
|\{ {\rm RP}^{{\rm UF}}(\bfx): \bfx 
\in C_0\setminus \bigcup \calN_m \}|,  
\]
i.e., the number of ranking patterns of unidimensional 
unfolding models with generic $m$-tuples 
such that $x_1<\cdots <x_m$. 
Then, by Lemma \ref{lm:RP=RP} we have  
\begin{equation} 
\label{eq:r0}
r_0(m)=\frac{|\ch(\calM_m)|}{m!}  
\end{equation} 
(Kamiya, Orlik, Takemura and Terao \cite{kott}). 

Now consider $r(m)$ in \eqref{eq:def-r(m)}. 
For $\bfx=(x_1,\ldots,x_m)\in 
\bbR^m\setminus \bigcup \calM_m$, 
define $-\bfx:=(-x_1,\ldots,-x_m)
\in \bbR^m\setminus \bigcup \calM_m$. 
Then, clearly we have ${\rm RP}^{{\rm UF}}(\bfx)=
{\rm RP}^{{\rm UF}}(-\bfx)$.  
On the other hand, for $C,C'\in \ch(\calM_m)$ 
such that $C' \ne \pm C \ (-C:=
\{ -\bfx: \bfx\in C\})$, we can easily see that  
${\rm RP}^{{\rm UF}}(\bfx)\ne 
{\rm RP}^{{\rm UF}}(\bfx')$ for $\bfx\in C$ 
and $\bfx'\in C'$.  
These two facts, together with Lemma \ref{lm:RP=RP}, 
yield the following theorem. 

\begin{theorem}
\label{th:r(m)n=1}
The number of ranking patterns of unidimensional 
unfolding models with generic $m$-tuples of objects is 
\[
r(m)=\frac{m!}{2}r_0(m)=\frac{|\ch(\calM_m)|}{2}, 
\quad m \ge 3. 
\] 
\end{theorem}

Let us define equivalence of ranking patterns by 
saying that two ranking patterns 
${\rm RP}^{{\rm UF}}(\bfx)$ and 
${\rm RP}^{{\rm UF}}(\bfx')$ 
are equivalent iff 
\begin{equation}
\label{eq:equiv}
{\rm RP}^{{\rm UF}}(\bfx)
=\sigma{\rm RP}^{{\rm UF}}(\bfx') \ \text{ for some } 
\sigma \in \dS_m, 
\end{equation}
where 
$\dS_m$ is the symmetric group on $m$ letters, 
consisting of all bijections: $[m]\to [m]$, 
and 
$\sigma{\rm RP}^{{\rm UF}}(\bfx'):=
\{ (\sigma(i_1)\cdots \sigma(i_m)): 
(i_1\cdots i_m)\in {\rm RP}^{{\rm UF}}(\bfx')\}$. 
We want to find the number of inequivalent 
ranking patterns. 

Let $r_{{\rm IE}}(m)$ be   
the number of inequivalent ranking patterns of 
unidimensional unfolding models with generic 
$m$-tuples of objects: 
\[
r_{{\rm IE}}(m):=|\{ [{\rm RP}^{{\rm UF}}(\bfx)]: 
\bfx \in \bbR^m\setminus \bigcup \calM_m\}|, 
\] 
where $[ \, \cdot \, ]$ stands for the equivalence class 
under the equivalence relation defined by 
\eqref{eq:equiv}. 
We will see that $r_{{\rm IE}}(m)$ is 
half of $r_0(m)$ for $m\ge 4$. 
Suppose we are given $\bfx=(x_1,\ldots,x_m)\in 
C_0\setminus \bigcup \calN_m$ with $m\ge 4$. 
Then $\bfx'=(x'_1,\ldots,x'_m):=
(-x_m,\ldots,-x_1)$ also lies in 
$C_0\setminus \bigcup \calN_m:$ 
$\bfx'\in C_0\setminus \bigcup \calN_m$. 
Moreover, since $m\ge 4$, four indices $1,2,m-1,m$ are 
all distinct and we have $x_{1m}
<x_{2, m-1}$ iff 
$x'_{1m}>x'_{2,m-1}$.   
This means ${\rm RP}^{{\rm UF}}(\bfx)\ne 
{\rm RP}^{{\rm UF}}(\bfx')$ by Lemma \ref{lm:RP=RP}. 
However, $[{\rm RP}^{{\rm UF}}(\bfx)]=
[{\rm RP}^{{\rm UF}}(\bfx')]$  
since ${\rm RP}^{{\rm UF}}(\bfx)= 
{\rm RP}^{{\rm UF}}(-\bfx)$. 
Next, it can be seen that 
any $\bfx''\in C_0\setminus \bigcup \calN_m$ such that 
${\rm RP}^{{\rm UF}}(\bfx'')\ne 
{\rm RP}^{{\rm UF}}(\bfx)$ and 
$[{\rm RP}^{{\rm UF}}(\bfx'')]=
[{\rm RP}^{{\rm UF}}(\bfx)]$ 
satisfies ${\rm RP}^{{\rm UF}}(\bfx'')=
{\rm RP}^{{\rm UF}}(\bfx')$. 
These arguments lead to the following theorem. 

\begin{theorem}
\label{th:inequiv-n=1}
The number of inequivalent ranking patterns of 
unidimensional unfolding models with generic 
$m$-tuples of objects is 
\[
r_{{\rm IE}}(m)=
\begin{cases}
r_0(3)=\frac{|\ch(\calB_3)|}{3!}=1 
& \text{ if $m=3$,} \\ 
\frac{r_0(m)}{2}
=\frac{|\ch(\calM_m)|}{2 \, \cdot \, m!} 
& \text{ if $m\ge 4$. } 
\end{cases}
\] 
\end{theorem}

So far, we have expressed the number of ranking patterns 
in terms of the number of chambers of an arrangement. 
We can use the finite field method 
(Athanasiadis \cite{ath-thesis, ath-96}, 
Crapo and Rota \cite{crr}, 
Kamiya, Takemura and Terao \cite{ktt-08, ktt-root, 
ktt-nc}, 
Stanley \cite[Lecture 5]{sta-07}) 
to calculate specific values of $r_0(m), \ m\le 10$:  
\begin{gather*} 
r_0(4) = 2, \ r_0(5) = 12, \ r_0(6) = 168, \ 
r_0(7) = 4680, \\ 
r_0(8) = 229386, \ r_0(9) = 18330206, \ 
r_0(10) = 2241662282. 
\end{gather*} 
The values of $r(m)$ for $m\le 8$ are given in 
Kamiya, Orlik, Takemura and Terao \cite{kott} along with 
the characteristic polynomials $\chi(\calM_m,t)$ of 
$\calM_m, \ m\le 8$. 
After \cite{kott}, the second author of the present 
paper, Takemura \cite{tak}, 
improved on Lemma 3.3 of \cite{kott} and 
calculated $\chi(\calM_9,t)$ and $r_0(9)$;  
later Ishiwata \cite{ish} obtained 
$\chi(\calM_{10},t)$ and 
$r_{0}(10)$ after an extensive computation.   
The characteristic polynomials found by them are: 
\begin{eqnarray*}
\chi(\calM_9, t) 
&=& t(t-1)(t^7 - 413t^6 + 73780t^5 - 7387310t^4 
+ 447514669t^3 \\ 
&& \qquad \qquad 
- 16393719797t^2 + 336081719070t - 2972902161600), \\ 
\chi(\calM_{10},t)
&=& t(t - 1)(t^8 - 674t^7 + 201481t^6 - 34896134t^5 
+ 3830348179t^4 \\ 
&& \qquad \qquad \qquad 
- 272839984046t^3 + 12315189583899t^2 \\ 
&& \qquad \qquad \qquad \qquad 
- 321989533359786t + 3732690616086600). 
\end{eqnarray*}

However, for large values of $m$, 
the finite field method 
is not feasible. 
We will provide simple upper and lower bounds for 
$r_0(m)$.  


%

%

\begin{theorem}\label{th:r0bound}
For all $m\geq 4$, we have
$$
2
\left( \frac{3}{4}\right)^{m-4} 
\left\{ (m-3)!\right\}^2
\leq 
r_0(m)<
\frac{2}{m!}\left\{ \frac{e m(m-1)^2}{8} \right\}^{m-2}.
$$
\end{theorem}

\begin{proof}
First, we derive the upper bound in the theorem. 

Define $H_0:=\{ (x_1,\ldots,x_m)\in \bbR^m: 
x_1+\cdots +x_m=0 \}$, and consider the 
essentialization 
(Stanley \cite[p.392]{sta-07}) 
$\calM_m^{0}:=\{ H \cap H_0: H \in \calM_m \}$ 
of $\calM_m$. 
Since $L(\calM_m^{0}) \cong L(\calM_m)$, we may 
consider the essential, central arrangement 
$\calM_m^0$ in $H_0 \ (\dim H_0=m-1)$ instead of 
$\calM_m$. 

Recall, in general, that  
$h$ hyperplanes divide $\bbR^d$ into
at most $\sum_{i=0}^d\binom{h}i\leq(eh/d)^d
=:c(h,d)$ 
chambers 
(see, e.g., \cite[Proposition 6.1.1]{mat} 
and \cite[Theorem 3.6.1]{mat-nes}). 
Thus, $\tilde{h}$ linear hyperplanes divide 
$\bbR^{\tilde{d}}$ into 
at most $2c(\tilde{h}-1,\tilde{d}-1)$ chambers.  

In our case, $\calM_m^0$ is central, 
so we can take 
$\tilde{h}=|\calM_m|=|\calB_m|+|\calN_m|
=\binom{m}2+3\binom{m}4
\le m(m-1)^2(m-2)/8 \ (m \ge 4)$ 
and $\tilde{d}=m-1$. 
Hence, we have 
\begin{eqnarray*}
|\ch(\calM^0_m)|
&\le& 2 c(\tilde{h}-1,\tilde{d}-1) \\ 
&\le& 2 \times \left\{ 
\frac{e \left( \frac{m(m-1)^2(m-2)}{8}-1\right)}{m-2} 
\right\}^{m-2} \\ 
&<& 2 \times \left\{ \frac{e m(m-1)^2}{8}\right\}^{m-2}. 
\end{eqnarray*}
This together with \eqref{eq:r0} 
and $|\ch(\calM_m)|=|\ch(\calM_m^0)|$ gives 
the upper bound of $r_0(m)$ in the theorem.   

Next, we will obtain the lower bound in the theorem. 

Let $\bfx=(x_1,\ldots,x_m)$, $x_1<\cdots<x_m$ be fixed. We add
one more object $y=x_m+2t$ $(t>0)$ to $\bfx$, and we will count
the number of ranking patterns arising from 
${\bf y}_t=(\bfx,y)$, $t>0$. 
Let $M=\{x_{ij}:1\leq i<j\leq m\}$ be the set of midpoints for $\bfx$, 
and 
$Y_t=\{x_{im}+t:1\leq i\leq m\}$ 
the set of midpoints of
$x_i \ (1\le i \le m)$ and $y$. 
Then $M\cup Y_t$ is the set of midpoints for ${\bf y}_t$. 
To guarantee all these midpoints are distinct, we require the following. 
First, by perturbing each $x_i$ without changing the ranking pattern 
of $\bfx$, we may assume that $x_1,\ldots,x_m$ are independent over 
$\mathbb Q$. Then we have $|M\cap Y_t|\leq 1$ for all $t>0$.
Next, let $T_0=\{t>0:|M\cap Y_t|=1\}$,
$T_1=(0,\infty)\setminus T_0$, and we only consider $t\in T_1$.
Then $M\cup Y_t$ is legal, i.e., all midpoints are distinct.

Now the crucial observation is as follows:
$|\{{\rm RP}^{\rm UF}({\bf y}_t):t\in T_1\}|=1+|T_0|$.
Moreover, we have $|T_0|=\sum_{i=1}^{m-1}|V_i|$, where
$V_i=\{v\in M:x_{im}<v\}$.
Using $|V_i|\geq m-1-i$ 
obtained by 
$V_i\supset\{x_{jm}:i<j<m\}$, 
we have
$$|\{{\rm RP}^{\rm UF}({\bf y}_t):t\in T_1\}|= 1+\sum_{i=1}^{m-1}|V_i| 
\geq 1+|V_1|+
\frac{(m-3)(m-2)}{2} 
=:N.$$
Namely, $N$ is a lower bound for the number of ranking patterns
arising from ${\bf y}_t, \ t\in T_1$. 

Applying exactly the same argument to $\bfx'=(-x_m,\ldots,-x_1)$ 
instead of $\bfx$, we see that the number of ranking patterns
arising from $(\bfx',-x_1+2t)$, $t>0$ 
(or equivalently, $(x_1-2t,\bfx)$, $t>0$) is at least
$N'=1+|V'_1|+(m-3)(m-2)/2$, where 
$|V_1'|=|\{u\in M:u<x_{1m}\}|=\binom{m}2-|V_1|-1$.
Notice that $N+N'=1+\binom{m}2+(m-3)(m-2)>(3/2)(m-2)^2$.
Therefore, by the averaging argument, we have
$$
\textstyle
r_0(m+1)\geq r_0(m)\times 
\frac{1}{2}(N+N') 
>\frac{3}{4}(m-2)^2\,r_0(m).
$$
So the induction starting from $r_0(4)=2$ gives the
desired lower bound.
\end{proof}


Let $\ell(m)$ and $u(m)$ be the lower and upper bounds in the
theorem, respectively. 
A computation shows 
$\{u(m)\}^{1/m}/m^2 \to e^2/8\approx 0.92$ and
$\{ \ell(m)\}^{1/m}/m^2\to 3/(4e^2)\approx 0.1$ 
as $m\to \infty$.
It would be interesting to prove (or disprove) the existence
of $\lim \{ r_0(m)\}^{1/m}/m^2$. 

Strangely enough, $r_0(m)=a(m)$ holds for $4\leq m\leq 7$,
where 
\[
a(m) := \frac{(m-2)\{ (m-2)^{m-3}-1\} \cdot (m-4)!}{m-3}, 
\]
but $r_0(8)>a(8)$, $r_0(9)>a(9)$, $r_0(10)>a(10)$. 
Also, $a(m)$ satisfies   
$\{ a(m)\}^{1/m}/m^2 \to 1/e\approx 0.37$. 
We mention that 
$a(m)/\{(m-3)!\}
=(m-2)\{(m-2)^{m-3}-1\}/(m-3)^2 
\ (m\ge 4)$ 
is the number of acyclic-function digraphs on $m-2$ vertices 
(Walsh \cite{wal}, OEIS id:A058128). 

Thrall \cite{thr} gave an upper bound $f(m)$ for $r_0(m)$: 
\[
f(m):=
\frac{
\{ \frac{m(m-1)}{2}\}! \, \prod_{i=1}^{m-2}i!}
{\prod_{i=1}^{m-1}(2i-1)!}.
\]
Here, $f(m)$ is the number of mappings 
$\{(i,j): 1\le i<j \le m\}\ni (i,j) \mapsto d(i,j)\in 
\{1,2,\ldots,m(m-1)/2\}$ satisfying the condition that 
$d(i,j)$ be increasing in $i$ for each fixed $j$ as well as 
increasing in $j$ for each fixed $i$. 
He obtained this number by considering a problem similar to 
that of counting the number of standard Young tableaux. 
Since for $\bfx=(x_1,\ldots,x_m)\in C_0\setminus \bigcup \calN_m$, 
the ranks $d_{\bfx}(i,j)$ of the midpoints 
$x_{ij}=(x_i+x_j)/2$ from left to right 
on the real line $\bbR$ meet this condition, $f(m)$ is an upper 
bound for $r_0(m)$.   
We can see our $u(m)$ satisfies $f(m)<u(m)$ for $m\leq 8$, 
$f(m)>u(m)$ for $m\geq 9$, and $u(m)=o(f(m))$. 
For $m$ such that $f(m)<u(m)$, we know the exact values 
$r_0(m)$ anyway, 
so the upper bound $u(m)$ based on arrangement theory may be 
said to be better than $f(m)$. 

We list the values of $r_0(m)$, $a(m)$, $f(m)$ 
and approximate values of $\ell(m)$, $u(m)$ 
for $m=4,\ldots,10$ in Table \ref{table:r0-f}. 
(For $\ell(m), \ m \le 9$, and 
$u(m), \  m\le 6$, we exhibit 
$\lceil \ell(m) \rceil$ and $\lfloor u(m) \rfloor$, 
respectively. 
For $\ell(10)$, we display 
$\lceil \ell(m) \times 10^{-4} \rceil \times 10^4$, 
and similarly using $\lfloor \, \cdot \, \rfloor$ 
for $u(m), \  m \ge 7$.) 

\begin{table}
\caption{$r_0(m), \, a(m), \, \ell(m), \, u(m), 
\, f(m), \ 4\le m\le 10$. 
}
\setlength{\tabcolsep}{4pt}
\begin{tabular}{crrrrr} \toprule 
{\footnotesize $m$} & {\footnotesize $r_0(m)$} & 
{\footnotesize $a(m)$} & {\footnotesize $\ell(m)$} & 
{\footnotesize $u(m)$} & {\footnotesize $f(m)$}
\\ \midrule 
{\footnotesize $4$} & {\footnotesize $2$} & 
{\footnotesize $2$} & 
{\footnotesize $2$} & 
{\footnotesize $12$} & 
{\footnotesize $2$} 
\\ 
{\footnotesize $5$} & {\footnotesize $12$} & 
{\footnotesize $12$} & 
{\footnotesize $6$} & 
{\footnotesize $334$} & 
{\footnotesize $12$} 
\\ 
{\footnotesize $6$} & {\footnotesize $168$} & 
{\footnotesize $168$} & 
{\footnotesize $41$} & 
{\footnotesize $18,744$} & 
{\footnotesize $286$} 
\\ 
{\footnotesize $7$} & {\footnotesize $4,680$} & 
{\footnotesize $4,680$} & 
{\footnotesize $486$} & 
{\footnotesize $1.82 \times 10^6$} & 
{\footnotesize $33,592$} 
\\ 
{\footnotesize $8$} & {\footnotesize $229,386$} & 
{\footnotesize $223,920$} & 
{\footnotesize $9,113$} & 
{\footnotesize $2.76 \times 10^8$} & 
{\footnotesize $23,178,480$} 
\\ 
{\footnotesize $9$} & {\footnotesize $18,330,206$} & 
{\footnotesize $16,470,720$} & 
{\footnotesize $246,038$} & 
{\footnotesize $6.06 \times 10^{10}$} & 
{\footnotesize $108,995,910,720$} 
\\ 
{\footnotesize $10$} & {\footnotesize $2,241,662,282$} & 
{\footnotesize $1,725,655,680$} & 
{\footnotesize $9.05 \times 10^6$} & 
{\footnotesize $1.81 \times 10^{13}$} & 
{\footnotesize $3,973,186,258,569,120$}    
\\ \bottomrule 
\end{tabular} 
\label{table:r0-f}
\end{table}

\subsection{Unfolding models of codimension one} 
\label{subsec:UMCO} 










In this subsection, we deal with the problem of counting 
the number of ranking patterns of unfolding models 
of codimension one: $n=m-2$
(i.e., when the restriction by dimension is 
weakest). 

First, let us forget the unfolding model for a while 
and consider the ranking patterns of braid slices. 

We begin by defining the ranking pattern of a braid 
slice. 
For  
\[
H_0=\{ x=(x_1,\ldots,x_m)^T\in \bbR^m: 
x_1+\cdots +x_m=0 \},  
\]
consider the essential arrangement 
\[
\calB_m^0:=\{ H\cap H_0: H\in \calB_m\}
\] 
in $H_0$, and write its chambers 
as  
\[
B_{i_1\cdots i_m}:=\{ x=(x_1,\ldots,x_m)^T\in H_0: 
x_{i_1} > \cdots >x_{i_m} \}
\in \ch(\calB_m^0) 
\]
for $(i_1\cdots i_m)\in \bbP_m$.  
Moreover, define a hyperplane 
\[
K_v:=\{ x\in H_0: v^Tx=1\}
\]
in $H_0$ for each 
$v\in \bbS^{m-2}:=\{ x\in H_0: \| x\|=1\}$. 
Now we call the subset 
\[
{\rm RP}(v):=\{ (i_1\cdots i_m)\in \bbP_m: 
K_v \cap B_{i_1\cdots i_m}\ne \emptyset \}, 
\quad v\in \bbS^{m-2}, 
\]
of $\bbP_m$ the ranking pattern of the braid 
slice by $K_v$.  

Next, let us define genericness of 
the braid slice as follows.  
For the all-subset arrangement 
(Kamiya, Takemura and Terao \cite{ktt-10}) 
\[
\calA_m:=\{ H_I: I \subseteq [m], \ 
|I|\ge 1 \}
\]
with $H_I:=\{ x=(x_1,\ldots,x_m)^T\in \bbR^m: 
\sum_{i\in I}x_i=0\}, \ \emptyset \ne I\subseteq [m]$, 
consider its restriction to $H_0=H_{[m]}$:  
\begin{gather*}
\calA^0_m:=\calA^{H_0}_m
=\{ H^0_I: I\subset [m], \ 1\le |I| \le m-1\}, \\ 
H^0_I:=H_I\cap H_0 \quad (1\le |I| \le m-1). 
\end{gather*}
Then define 
\[
\calV:=(H_0\setminus \bigcup \calA_m^0)\cap \bbS^{m-2}. 
\]
We will say $v\in \bbS^{m-2}$, or the braid slice by 
$K_v$, is generic if $v\in \calV$.    


Now, we will see that the set of 
ranking patterns 
${\rm RP}(v)$ for generic $v$'s 
is in one-to-one correspondence with the 
set of 
chambers of $\calA_m^0$.  
Write $\calV$ as 
$\calV=\bigsqcup_{D\in \Dch(\calA_m^0)}D$ 
(disjoint union), where  
\[
\Dch(\calA_m^0):=\{ D=\tilde{D}\cap \bbS^{m-2}: 
\tilde{D}\in \ch(\calA_m^0) \},  
\]
which clearly is in one-to-one correspondence with 
$\ch(\calA_m^0)$. 
Then, we can prove 
(Kamiya, Takemura and Terao \cite{ktt-10}) 
that there is a bijection from $\Dch(\calA_m^0)$ to 
$\{ {\rm RP}(v): v\in \calV \}$ given by 
\begin{equation}
\label{eq:bijection-DtoRP}
\Dch(\calA_m^0)\ni D \mapsto {\rm RP}(v), \ v\in D. 
\end{equation}
Hence, 
\[
{\rm RP}_D:={\rm RP}(v) 
\text{ for } v\in D \in \Dch(\calA_m^0) 
\] 
is well-defined, and 
the mapping $\Dch(\calA_m^0) \to 
\{ {\rm RP}_D: D\in \Dch(\calA_m^0)\}
=\{ {\rm RP}(v): v\in \calV \}: 
D \mapsto {\rm RP}_D$ is bijective.  

Let us get back to the unfolding model and consider 
the ranking pattern of the unfolding 
model of codimension one. 

Suppose we are given $x_1,\ldots,x_m \in \bbR^n$ with 
$n=m-2\ge 1$. 
We assume $x_1,\ldots,x_m$ are generic in the sense 
that they satisfy (A1) and (A2) in 
Section \ref{sec:NAR}. 
We call the unfolding model with such $x_1,\ldots,x_m
\in \bbR^{m-2}$ the unfolding model of codimension one 
(for the reason stated 
below).  
In addition, we will 
assume without loss of generality 
that $x_1,\ldots,x_m$ are taken so that 
$\sum_{i=1}^mx_i=0, \ \sum_{i=1}^m\| x_i\|^2/m=1$. 

We will 
see that the ranking pattern of 
the unfolding model of codimension one 
with $m$-tuple $(x_1,\ldots,x_m)$: 
\begin{eqnarray}
{\rm RP}^{{\rm UF}}(x_1,\ldots,x_m)
&=& \{ (i_1\cdots i_m)\in \bbP_m: 
\| y-x_{i_1}\|<\cdots <\| y-x_{i_m}\| 
\notag \\ 
&& \qquad \qquad \qquad \qquad \qquad \qquad 
\text{ for some } y \in \bbR^{m-2}\} 
\label{eq:||<||}
\end{eqnarray}
can be expressed as the ranking pattern of 
a braid slice. 

Define 
\begin{eqnarray*}
W &=& {\rm W}(x_1,\ldots,x_m)=(w_1,\ldots,w_{m-2}):=
\begin{pmatrix}
x_1^T \\ 
\vdots \\ 
x_m^T
\end{pmatrix}
\in {\rm Mat}_{m\times (m-2)}(\bbR), \\ 
u &=& {\rm u}(x_1,\ldots,x_m):=-\frac{1}{2}
\begin{pmatrix}
\| x_1\|^2-1 \\ 
\vdots \\ 
\| x_m\|^2-1 
\end{pmatrix}
\in \bbR^m, 
\end{eqnarray*}
where ${\rm Mat}_{m\times (m-2)}(\bbR)$ 
denotes the set of 
$m\times (m-2)$ matrices with real entries. 
For the affine map 
$\kappa: \bbR^{m-2} \to \bbR^m$ defined 
by $\kappa(y):=Wy+u, \ y\in \bbR^{m-2}$, 
consider the image 
$K:=\im \kappa =\{ k(y): y\in \bbR^{m-2}\}$ 
of $\kappa$.  
Then we have 
\[
K=u+{\rm col}\,W \subset H_0, 
\]  
where ${\rm col}\,W$ stands for the column space of $W$. 
Using this $K$, 
we can easily see that 
${\rm RP}^{{\rm UF}}(x_1,\ldots,x_m)$ in 
\eqref{eq:||<||} can be expressed as 
\begin{equation}
\label{eq:RP-KcapC}
{\rm RP}^{{\rm UF}}(x_1,\ldots,x_m)
=\{ (i_1\cdots i_m)\in \bbP_m: 
K\cap B_{i_1\cdots i_m}\ne \emptyset \}. 
\end{equation}
We have $\dim K=\dim H_0-1$ and $u\notin {\rm col}\,W$ 
by (A1) and (A2), respectively. 
That is, $K$ is an affine hyperplane of $H_0$. 
For this reason, we called the unfolding model 
with generic 
$x_1,\ldots,x_m\in \bbR^{m-2}$ the unfolding model of 
codimension one. 

Write the affine hyperplane $K\subset H_0$ as 
\[
K=K_{\tilde{v}}:=\{ x\in H_0: \tilde{v}^Tx
=\| \tilde{v} \|^2\} 
\]   
using the orthogonal projection 
of $u\in H_0$ on  
$({\rm col}\,W)^{\perp}:=\{ x \in 
H_0: 
x^T W=0\}$:   
\[
\tilde{v}:=\tilde{{\rm v}}(x_1,\ldots,x_m)
=u-{\rm proj}_{{\rm col}\,W}(u),  
\quad u={\rm u}(x_1,\ldots,x_m), 
\] 
where ${\rm proj}_{{\rm col}\,W}$ denotes the 
orthogonal projection on ${\rm col}\,W$.   
Noting $\tilde{v}\ne 0$, we can represent 
\eqref{eq:RP-KcapC} as 
\begin{gather}
{\rm RP}^{{\rm UF}}(x_1,\ldots,x_m)
= \{ (i_1\cdots i_m)\in \bbP_m: 
K_{{\rm v}(x_1,\ldots,x_m)}
\cap B_{i_1\cdots i_m}\ne \emptyset \}, 
\label{eq:RPUF-v} \\ 
{\rm v}(x_1,\ldots,x_m):=
\frac{1}{\|\tilde{v}\|}\tilde{v}\in \bbS^{m-2}, 
\notag 
\end{gather}
in terms of $K_{{\rm v}(x_1,\ldots,x_m)}
=\{ x\in H_0: {\rm v}(x_1,\ldots,x_m)^Tx=1 \}$ 
instead of $K=K_{\tilde{v}}$. 
The right-hand side of \eqref{eq:RPUF-v} is 
the ranking pattern of the braid slice by 
$K_{{\rm v}(x_1,\ldots,x_m)}$: 
${\rm RP}({\rm v}(x_1,\ldots,x_m))$. 
Besides, it can be seen that 
${\rm v}(x_1,\ldots,x_m)\in \calV$. 

\begin{prop}[Kamiya, Takemura and Terao \cite{ktt-10}] 
\label{prop:RP=RP}
For generic $x_1,\ldots,x_m\in \bbR^{m-2}$, we have 
${\rm v}(x_1,\ldots,x_m)\in \calV$ and 
$$
{\rm RP}^{{\rm UF}}(x_1,\ldots,x_m)=
{\rm RP}({\rm v}(x_1,\ldots,x_m)). 
$$ 
\end{prop}

Proposition \ref{prop:RP=RP} 
and bijection \eqref{eq:bijection-DtoRP} 
tell us that 
in order to find the number of ranking patterns of 
unfolding models of codimension one, 
we need to study the image 
of the mapping 
${\rm v}: \{ (x_1,\ldots,x_m): 
x_1,\ldots,x_m\in \bbR^{m-2} \text{ are generic}\} 
\to \calV=\bigsqcup_{D\in \Dch(\calA_m^0)}D, \ 
(x_1,\ldots,x_m) \mapsto {\rm v}(x_1,\ldots,x_m)$.  
In their main theorem (Theorem 4.1), Kamiya, Takemura 
and Terao \cite{ktt-10} proved that the image 
$\im {\rm v}$ is given by 
\begin{equation}
\label{eq:im-v}
\im {\rm v}
=\calV_2 \sqcup D_1 \sqcup \cdots \sqcup D_m
=\calV \setminus ((-D_1)\sqcup \cdots \sqcup (-D_m)), 
\end{equation}
where 
\begin{eqnarray*} 
\calV_2 &:=& 
\{ v=(v_1,\ldots,v_m)^T\in \calV: 
\text{ $v_j>0$ for at least two $j \in [m]$ and} \\ 
&& \qquad \qquad \qquad \qquad \qquad \qquad \quad 
\text{$v_k<0$ for at least two $k\in [m]$} \} 
\end{eqnarray*}
and 
\begin{eqnarray*}
D_i &:=& \{ v=(v_1,\ldots,v_m)^T\in \calV: 
v_i>0, \ v_j<0 \ (j\ne i)\}\in \Dch(\calA_m^0), \\ 
-D_i &:=& \{ -v: v\in D_i \} \\ 
&=& \{ v=(v_1,\ldots,v_m)^T\in \calV: 
v_i<0, \ v_j>0 \ (j\ne i)\}
\in \Dch(\calA_m^0)
\end{eqnarray*}
for $i\in [m]$. 

By Proposition \ref{prop:RP=RP} 
and 
$\im {\rm v}$ in \eqref{eq:im-v}, 
we obtain the number of ranking patterns of 
unfolding models of codimension one, 
which is denoted by 
\[
q(m):=|\{ {\rm RP}^{\rm UF}(x_1,\ldots,x_m): 
\text{ generic } x_1,\ldots,x_m \in \bbR^{m-2}\}|. 
\]

\begin{theorem}[Kamiya, Takemura and Terao \cite{ktt-10}]
\label{th:q(m)=||-m}
The number $q(m)$ of ranking patterns of 
unfolding models of codimension one is given by 
\[
q(m)=|\ch(\calA_m^0)|-m. 
\]
\end{theorem}

Kamiya, Takemura and Terao \cite[Lemma 5.3]{ktt-10} 
obtained the characteristic polynomials 
$\chi(\calA_m^0,t)$ of $\calA_m^0$ for $m\le 8$ 
by the finite field method.  
Then $q(m)$ can be calculated by 
$q(m)=(-1)^{m-1}\chi(\calA_m^0,-1)-m$: 
\begin{gather*}
q(3) = 3, \ q(4) = 28, \ q(5) = 365, \\ 
q(6) = 11286, \ q(7) = 1066037, \ q(8) = 347326344  
\end{gather*}
(
\cite[Corollary 5.5]{ktt-10}). 

We end this subsection by looking at the problem of 
finding the number of inequivalent ranking patterns 
of unfolding models of codimension one. 

In \eqref{eq:equiv}, we defined equivalence of 
ranking patterns of unidimensional unfolding models.  
We define equivalence of ranking patterns of unfolding 
models of codimension one in an obvious similar manner. 
At the moment,  
we can only give an upper bound for the number 
$q_{{\rm IE}}(m)$ of inequivalent ranking patterns 
of unfolding models of codimension one: 
\begin{equation}
\label{eq:qIE-le}
q_{{\rm IE}}(m) 
\le \frac{|\ch(\calA_m^0 \cup \calB_m^0)|}{m!}-1 
= |\Dch^{1\cdots m}(\calA_m^0)|-1 
= |\Dch_2^{1\cdots m}(\calA_m^0)|+1
\end{equation}
for $m\ge 3$ 
(Kamiya, Takemura and Terao \cite{ktt-10}), where 
$\Dch^{1\cdots m}(\calA_m^0)
:= \{ D\in \Dch(\calA_m^0): 
D \cap B_{1\cdots m}\ne \emptyset \}$ 
and 
$\Dch^{1\cdots m}_2(\calA_m^0)
:= \{ D\in \Dch(\calA_m^0): 
D \subset \calV_2, \ 
D \cap B_{1\cdots m}\ne \emptyset \}
=\Dch^{1\cdots m}(\calA_m^0)\setminus \{ D_1, -D_m\}$. 
%
%
It is shown in \cite{ktt-10}, however,  
that the upper bound 
in \eqref{eq:qIE-le} is actually the exact number 
for $m\le 6$. 
The specific values are 
\[
q_{{\rm IE}}(3)=1, \ q_{{\rm IE}}(4)=3, \ 
q_{{\rm IE}}(5)=11, \ q_{{\rm IE}}(6)=55   
\]
(\cite[Subsection 6.2]{ktt-10}). 

{\bf Open problem:} 
Does the upper bound in \eqref{eq:qIE-le} 
agree with the exact number $q_{{\rm IE}}(m)$ 
for all $m$? 

\ 

{\bf Acknowledgments.} 
The authors are very grateful to an anonymous referee 
of Advanced Studies in Pure Mathematics 
for his/her 
valuable comments on an earlier version of this paper. 
The first author is 
supported by JSPS KAKENHI (19540131). 
The second author is 
supported by JST CREST. 
The last author is supported by JSPS KAKENHI (20340022).


\noindent
Hidehiko Kamiya \\ 
{\it 
Graduate School of Economics \\ 
Nagoya University \\ 
Nagoya 464-8601 \\ 
Japan }\\
{\it E-mail address:} {\tt kamiya@soec.nagoya-u.ac.jp} 

\bigskip
\noindent
Akimichi Takemura \\ 
{\it Graduate School of Information Science and Technology \\ 
University of Tokyo \\ 
Tokyo 113-0033 \\ 
Japan}\\
{\it E-mail address:} {\tt takemura@stat.t.u-tokyo.ac.jp} 

\bigskip
\noindent
Norihide Tokushige \\ 
{\it 
College of Education \\ 
Ryukyu University \\ 
Nishihara, Okinawa 903-0213 \\ 
Japan}\\
{\it E-mail address:} {\tt hide@edu.u-ryukyu.ac.jp}

\end{document}